\documentclass{birkjour}
\usepackage{color}
 \newtheorem{thm}{Theorem}
 
 \newtheorem{lem}[thm]{Lemma}
 
 \theoremstyle{definition}
 
 \theoremstyle{remark}

 \numberwithin{equation}{section}

\newcommand{\N}{{\mathbb N}}

\newcommand{\supp}{\operatorname{supp}}
\begin{document}

%
%
%
%
%
%
%
%
%

\title[Note on the resonance method for the Riemann zeta function]{Note on the resonance method \\ for the Riemann zeta function}

\author[Andriy Bondarenko]{Andriy Bondarenko}

\address{%
Department of Mathematical Sciences \\ Norwegian University of Science and Technology \\ NO-7491 Trondheim \\ Norway}

\email{andriybond@gmail.com}

\author[Kristian Seip]{Kristian Seip}
\address{Department of Mathematical Sciences \\ Norwegian University of Science and Technology \\ NO-7491 Trondheim \\ Norway}
\email{seip@math.ntnu.no}
\thanks{Research supported in part by Grant 227768 of the Research Council of Norway. }




\dedicatory{To the memory of Victor Havin}

\begin{abstract}
We improve Montgomery's $\Omega$-results for $|\zeta(\sigma+it)|$ in the strip $1/2<\sigma<1$ and give in particular lower bounds for the maximum of $|\zeta(\sigma+it)|$ on $\sqrt{T}\le t \le T$ that are uniform in $\sigma$.
We give similar lower bounds for the maximum of $|\sum_{n\le x} n^{-1/2-it}|$ on intervals of length much larger than $x$.  We rely on our recent work on lower bounds for maxima of $|\zeta(1/2+it)|$ on long intervals, as well as work of Soundararajan, G\'{a}l, and others. The paper aims at displaying and clarifying the conceptually different combinatorial arguments that show up in various parts of the proofs. \end{abstract}

\maketitle
\section{Introduction}
Soundararajan \cite{S}  and Hilberdink \cite{Hi} presented independently slightly different versions of a technique, known as the resonance method, for detecting large values of the Riemann zeta function $\zeta(s)$. In our recent paper \cite{BS2}, we used Soundararajan's version of this method and the construction of a special multiplicative function to show that
\begin{equation} \label{maxduke} \max_{\sqrt{T}\le t \le T} \left|\zeta\Big(\frac{1}{2}+it\Big)\right| \ge \exp\left(\Big(\frac{1}{\sqrt{2}}+o(1)\Big)\sqrt{\frac{\log T \log\log\log T}{\log\log T}}\right)\end{equation}
when $T\to \infty.$
This gave an improvement by a power of  $\sqrt{\log\log\log T}$ compared with previously known estimates \cite{BR,S}.

In this note, we will apply the resonance method to two closely related problems, namely to find large values of respectively $\zeta(\sigma+it)$ for $1/2<\sigma<1$ and the partial sum $\sum_{n\le M}n^{-1/2-it}$ on certain long intervals (depending on $M$). We will find uniform lower bounds on the maximum  in the strip $1/2<\sigma<1$ and show in particular that the bound on the right-hand side of \eqref{maxduke} (with $1/\sqrt{2}$ replaced by a different constant) holds as far as $1/\log\log T$ to the right of the critical line.

Before proceeding to the details of these new results, we would like to comment on the relation between our subject and Hardy spaces, the presumed topic of the present volume. As outlined in \cite{SS},
our construction of resonators originates in Bohr's several complex variables perspective of Dirichlet series and our recent work on Hardy spaces of Dirichlet series.  The present paper can thus be viewed as an outgrowth of the remarkably rich subject of Hardy spaces and more specifically of a branch of it that interacts with number theory. Moreover, one may interpret our theorem on partial sums (Theorem~\ref{part} below) as dealing with a well known type of problem in complex analysis, namely how small the maximal size of an analytic function can be on a set of uniqueness that in some sense is ``small''.  For further information about Hardy spaces of Dirichlet series and connections with number theory, we refer to the survey paper \cite{SS} and the monograph \cite{QQ}.

\section{Statement of main results}
A less precise version of the following result was stated without proof in \cite{BS2}.

\begin{thm} \label{extreme}
There exists a positive and continuous function $\nu(\sigma)$ on  $(1/2,1)$, bounded below by $1/(2-2\sigma)$, with the asymptotic behavior
 \[ \nu(\sigma)=\begin{cases} (1-\sigma)^{-1}+O(|\log(1-\sigma)|), & \sigma \nearrow 1 \\
(1/\sqrt{2}+o(1))\sqrt{|\log(2\sigma-1)|}, & \sigma\searrow 1/2, \end{cases} \]
and such that the following holds.
If  $T$ is sufficiently large, then for $1/2+1/\log\log T \le \sigma \le 3/4$,
\begin{equation}\label{intermediate1} \max_{t\in[\sqrt{T},T]} \left|\zeta\Big(\sigma+it\Big)\right| \ge \exp\left(\nu(\sigma)\frac{(\log T)^{1-\sigma}}{(\log\log T)^\sigma}\right) \end{equation}
and for $3/4\le \sigma \le 1-1/\log\log T$,
\begin{equation}\label{intermediate2} \max_{t\in[T/2,T]} \left|\zeta\Big(\sigma+it\Big)\right| \ge \log\log T\exp\left(c+\nu(\sigma)\frac{(\log T)^{1-\sigma}}{(\log\log T)^\sigma}\right), \end{equation}
with $c$ an absolute constant independent of $T$.
 \end{thm}
To place this result in context, we recall Levinson's classical estimate\footnote{This result was later improved by Granville and Soundararajan \cite{GS} who managed to add a positive term of size $\log\log\log T$ on the right-hand side of  \eqref{Norman}.} \cite{L}
\begin{equation}\label{Norman} \max_{1\le t \le T} |\zeta(1+it)|\ge e^{\gamma}\log\log T + O(1), \end{equation}
where $\gamma$ is the Euler--Mascheroni constant. We now observe that Theorem~\ref{extreme} gives a ``smooth'' transition between the two endpoint cases \eqref{Norman} and \eqref{maxduke}. The factor $\log\log T$ on the right-hand side of \eqref{intermediate2} is only needed for $\sigma$ close to the right endpoint
$1-1/\log\log T$, to get the transition to Levinson's estimate. Theorem~\ref{extreme} gives a notable improvement of a classical estimate of Montgomery \cite{M} for the range $1/2<\sigma <1$. See \cite{RS} and the discussion in \cite{BS2} for the best estimates known previously.

The choice of intermediate abscissa $\sigma=3/4$ is somewhat
arbitrary (any fixed $\sigma_0$, $1/2<\sigma_0<1$ would do), and we could have shortened the interval in \eqref{intermediate1} (depending on $\sigma$). Indeed, the precise statement of Theorem~\ref{extreme} is a tradeoff between conveying the main point of the transition between the two endpoint cases and keeping the technicalities  reasonably simple.

We have refrained from making a precise statement about sharp estimates in the short intervals
$[1/2, 1/2+1/ \log\log T]$ and $[1-1/\log\log T, 1]$, although our method would certainly allow us to do it. The main point is that the order of magnitude of the respective endpoint estimates persists in these intervals. It may seem surprising that these intervals are as long as $1/\log\log T$ on either side. We will see in the course of the proof that this can be attributed to the resonance method's selection of smooth numbers\footnote{The smoothness (sometimes called the friability) of a positive integer $n$ is measured by the largest prime $p$ dividing $n$. The smaller this prime is, the smoother the number is.} in the construction of resonating Dirichlet polynomials.

In our proof of Theorem~\ref{extreme}, we use the approximate formula
\begin{equation}\label{approx} \zeta(\sigma+i t)= \sum_{n\le x} n^{-\sigma-it} - \frac{x^{1-\sigma-it}}{1-\sigma-it}+O(x^{-\sigma}), \end{equation}
which holds uniformly in the range $\sigma\ge \sigma_0>0$, $|t|\le x$ (see \cite[Theorem 4.11]{T}). This means that detecting large values of $\zeta(\sigma+i t)$ for $1/2\le \sigma \le 1$ and $|t|\le T$ is mainly a question about finding large values of the Dirichlet polynomial $\sum_{n\le T} n^{-\sigma-it}$ for $|t|\le T$.

We find it to be of interest to see what we get when we look for large values of just the partial sum itself on longer intervals. Thus we remove the a priori restriction on the length of the interval forced upon us by the approximate formula  \eqref{approx}.  We will only consider $\sigma=1/2$ and introduce the notation
\[
D_M(t)=\sum_{n\le M}n^{-1/2-it}.
\]
\begin{thm}\label{part}
Suppose that $c$, $0<c<1/2$, is given. If $T$ is sufficiently large and $M\ge \exp\big(e\sqrt{\log T \log\log T \log\log\log  T/2}\big)$,
then
\[
\max_{t\in[\sqrt{T},T]}|D_M(t)|\ge \exp\left(c\sqrt{\frac{\log T \log\log\log T}{\log\log T}}\right).
\]
\end{thm}
This theorem gives information about the precision of the resonance method as well as its limitations. We notice that the global maximum satisfies
\[ \|D_M\|_\infty:=\max_{t} |D_M(t)| \sim \sqrt{M}, \]
and hence we see that  our method gives us that when $M$ takes the minimal value $\exp\big(e\sqrt{\log T \log\log T \log\log\log  T/2}\big)$, the maximum on  $[\sqrt{T},T]$ is at least $\|D_M\|_\infty^{\eta/\log\log M}$ for some positive number $\eta$. This means that the value of the maximum ``predicted'' by the resonance method is at most a power of order $1/\log\log T$ (or equivalently of order $1/\log\log M$) off the true maximum (whatever it is). On the other hand,  while we know that $|D_M(t)|$ ``eventually'' will come arbitrarily close to the absolute maximum, the interval $[\sqrt{T},T]$ is of course far too short for us to guarantee, by general considerations,  that we get anywhere near $\| D_M \|_\infty$. Hence the resonance method could in fact be considerably more precise than what we can safely conclude that it is in this case.


A word on notation, before we turn to a general discussion of the resonance method and the proofs of our theorems: We will use the shorthand notation $\log_2 x:=\log\log x$ and $\log_3 x:=\log\log\log x$.
\section{The resonance method---general considerations}

The basic idea of either versions of the  resonance method is to identify a special Dirichlet polynomial
\[ R(t)=\sum_{m\in \mathcal{N}} r(m) m^{-it} \]
that ``resonates well'' with the object at hand, which in our case is the partial sum $\sum_{n\le x} n^{-\sigma-it}$  on a given interval. The precise meaning of this is that the integral of $|R(t)|^2$ (mollified by multiplication by a suitable smooth bump function) times $\sum_{n\le x} n^{-\sigma-it}$ is as large as possible, given that the coefficients $r(m)$ have a fixed square sum and also subject to whatever a priori restrictions we need to put on the set of integers $\mathcal{N}$. The method will not only produce large values, but also give information about which of the terms in $\sum_{n\le x} n^{-\sigma-it}$ that contribute in an ``essential'' way.

The technicalities will differ considerably depending on $\sigma$, for reasons that will become clear below. In our study of $\zeta(\sigma+it)$ in the range $3/4\le \sigma \le 1$, we will use Soundararajan's original method. This means that we choose a smooth function $\Psi$ compactly supported in $[1/2,1]$, taking values in the interval $[0,1]$ with
$\Psi(t)=1$ for $5/8\le t \le 7/8$. We define
\begin{align} \label{M1} M_1(R,T)& :=\int_{-\infty}^\infty |R(t)|^2 \Psi\Big(\frac{t}{T} \Big) dt, \\ \label{M2}
M_2(R,T)& :=\int_{-\infty}^\infty  \zeta(\sigma+it) |R(t)|^2 \Psi\Big(\frac{t}{T} \Big) dt.
\end{align}
Then
\begin{equation} \label{plain} \max_{T/2\le t \le T} \big|\zeta(\sigma+it)\big| \ge \frac{|M_2(R,T)|}{M_1(R,T)}, \end{equation}
and the goal is therefore to maximize the ratio on the right-hand side of \eqref{plain}. We require that
$\max \mathcal{N}\le T^{1-\varepsilon}$ for some fixed $\varepsilon$, $0<\varepsilon<1$, and get by straightforward computations (see \cite[pp. 471--472]{S}) that
\begin{equation} \label{Sound1}
M_1(R,T)=T\hat{\Psi}(0)\big(1+O(T^{-1})\big) \sum_{n\in \mathcal{N}} |r(n)|^2
\end{equation}
and
\begin{equation} \label{Sound2}
M_2(R,T)=T\hat{\Psi}(0)\sum_{n\in \mathcal{N}, mk=n}\frac{r(m) \overline{r(n)}}{k^{\sigma}}+O(T^{1-\sigma}\log T) \sum_{n\in \mathcal{N}} |r(n)|^2.
\end{equation}
Hence the problem of estimating the right-hand side of \eqref{plain} boils down to finding out how large the ratio
\begin{equation}\label{maxim} \sum_{n\in \mathcal{N}, mk=n}\frac{r(m) \overline{r(n)}}{k^{\sigma}}\Big/\sum_{n\in \mathcal{N}} |r(n)|^2 \end{equation}
can be under the a priori restriction that $\max \mathcal{N}\le T^{1-\varepsilon}$. This problem was solved in \cite[Theorem 2.1]{S} for $\sigma=1/2$.

As shown in \cite{BS2}, we can do better when $\sigma=1/2$ by removing the a priori restriction that $\max \mathcal{N}\le T^{1-\varepsilon}$, and this is also true when $\sigma$ is not too close to $1$. If we again manage to reduce the problem to that of maximizing a ratio like the one in \eqref{maxim}, then we clearly are in a better position. However, arriving at such an optimization problem is less straightforward, mainly because more terms will contribute in either of the sums representing respectively $M_1(R,T)$ and $M_2(R,T)$. An additional problem is that sets of integers $\mathcal{N}$ involved in making expressions like $\eqref{maxim}$ large typically enjoys a multiplicative structure, while estimating sums like those representing $M_1(R,T)$ and $M_2(R,T)$ requires some ``additive control''. We will now present the remedies, introduced in \cite{BS2}, for  getting around these problems.

We begin with the problem of ``additive control". We go ``backwards'' and start from the problem of maximizing
\begin{equation}\label{maximm} \sum_{n\in \mathcal{M}, mk=n}\frac{f(m) \overline{f(n)}}{k^{\sigma}}\Big/\sum_{n\in \mathcal{M}} |f(n)|^2 \end{equation}
for a suitable set $\mathcal{M}$ and arithmetic function $f(n)$ under the condition that $|\mathcal{M}|\le N$. We then extract the resonating Dirichlet polynomial from a solution (or approximate solution) to this problem as follows, assuming as we may that $f(n)$ is nonnegative. Following an idea from \cite{A}, we let $\mathcal{J}$ be the set of integers $j$ such that
\[ \Big[(1+T^{-1})^j,(1+T^{-1})^{j+1}\Big)\bigcap \mathcal{M} \neq \emptyset,  \]
and let $m_j$ be the minimum of  $\big[(1+T^{-1})^j,(1+T^{-1})^{j+1}\big)\bigcap \mathcal{M}$ for $j$ in $\mathcal{J}$. Then set
\[ \mathcal{N}:= \big \{ m_j: \ j\in \mathcal{J} \big\}\]
and
\begin{equation} \label{average} r(m):= \left(\sum_{n\in \mathcal{M}, 1-T^{-1}(\log T)^2 \le n/m \le 1+T^{-1}(\log T)^2} f(n)^2\right)^{1/2} \end{equation}
for every $m$ in $\mathcal{N}$. Note that plainly $|\mathcal{N}|\le |\mathcal{M}|\le N$. By taking the local $\ell^2$ average as in \eqref{average}, we get a precise relation between $f(n)$ and $r(n)$, while at the same time we get the desired ``additive control'' because each of the intervals $\big[(1+T^{-1})^j,(1+T^{-1})^{j+1}\big)$ contains at most one integer from  $\mathcal{N}$.

We turn next to the counterparts to \eqref{M1} and \eqref{M2}. We consider now a longer interval of the form $[T^\beta, T]$ with $0<\beta<1$; it will be convenient for us to fix once and for all $\beta=1/2$. Moreover, we use the Gaussian $\Phi(t):=e^{-t^2/2}$ as mollifier. Our replacements for \eqref{M1} and \eqref{M2} are then, respectively,
\begin{align} \widetilde{M}_1(R,T)& :=\int_{\sqrt{T}\le |t|\le T} |R(t)|^2 \Phi\Big(\frac{\log T}{T} t\Big) dt, \nonumber \\
\widetilde{M}_2(R,T)& :=\int_{\sqrt{T}\le |t| \le T}  \zeta(\sigma+it) |R(t)|^2 \Phi\Big(\frac{\log T}{T} t\Big) dt, \label{M22}
\end{align}
and we get that
\begin{equation} \label{plain2} \max_{\sqrt{T}\le t \le T} \big|\zeta(\sigma+it)\big| \ge \frac{|\widetilde{M}_2(R,T)|}{\widetilde{M}_1(R,T)}. \end{equation}

We state first the estimate for $\widetilde{M}_1(R,T)$ obtained in \cite[Formula (22)]{BS2}. This is a matter of direct computation based on the definitions given above.

\begin{lem}\label{m1lemma}
There exists an absolute constant $C$ such that
\begin{equation} \label{m1f} \widetilde{M}_1(R,T) \le C T (\log T)^3 \sum_{n\in \mathcal{M}} f(n)^2. \end{equation}
\end{lem}

To estimate \eqref{M22}, we extend the integral to the whole real line, so that we can take advantage of the fact that the Fourier transform $\widehat{\Phi}$ of $\Phi$ is positive. We chose the larger set $\sqrt{T}\le |t|\le T$ and a different dilation factor of the mollifier ($(\log T)/T$ instead of $1/T$), because these choices allow us to get the control we need of the integral over the complementary set. Indeed, the estimation for $|t|\ge T$ is trivial because of the rapid decay of the Gaussian and our choice of dilation factor, while the following estimate takes care of the  interval $|t|\le \sqrt{T}$: For arbitrary numbers $\lambda>0$, $0<\beta <1$, and $0<\sigma<1$, we have
\begin{equation} \label{bound}
\Big| \sum_{1\le n\le M}n^{-\sigma}\int_{-T^\beta}^{T^{\beta}}\left(\frac{\lambda}{n}\right)^{it}\Phi\Big(\frac{\log T}{T} t\Big) dt\Big|
\le C \max\left(T^{\beta},M^{1-\sigma} \log M\right),
\end{equation}
where the constant $C$ is independent of $\lambda$, $\beta$, $\sigma$. This is proved by making a minor adjustment of the proof of \cite[Lemma 4]{BS2}, which deals only with the case $\sigma=1/2$. Doing the same computations as in \cite{BS2}, relying crucially on the positivity of $\widehat{\phi}$, we arrive at the following lemma (see formula (14) in \cite{BS2}):

\begin{lem}\label{m2lem} Suppose $1/2\le \sigma \le 1$ and $|\mathcal{M}|\le \sqrt{T}$. There exist absolute positive constants $c, C$ such that
\[ \widetilde{M}_2(R,T)\ge c \left(\frac{T}{\log T} \sum_{n\in\mathcal{M}, mk=n, k\le T}\frac{f(n)f(m)}{k^\sigma}- C T (\log T)^4  \sum_{n\in \mathcal{M}}f(n)^2\right). \] 
\end{lem}

The powers of $\log T$ are harmless if $\sigma\le \sigma_0<1$ for some fixed $\sigma_0$, but they make this lemma useless when $\sigma$ is close to $1$ since the lower bound in \eqref{Norman} is of order $\log_2 T$. This is why we need  both versions of the resonance method when  considering the whole range
$1/2+1/\log_2 T\le \sigma \le 1-1/\log_2 T$.

The resonance method for the partial sum problem yields the same bounds, up to an obvious modification
of the indices in the summation in Lemma~\ref{m2lem}. Indeed, defining
\[
\widetilde{\widetilde{M}}_2(R,T) :=\int_{\sqrt{T}\le |t| \le T}  D_M(t)  |R(t)|^2 \Phi\Big(\frac{\log T}{T} t\Big) dt,
\]
we get:
\begin{lem} \label{m2lemD} Suppose that $|\mathcal{M}|\le \sqrt{T}$. There exist absolute positive constants $c, C$ such that
\[ \widetilde{\widetilde{M}}_2(R,T)\ge c \left(\frac{T}{\log T} \sum_{n\in\mathcal{M}, mk=n, k\le M}\frac{f(n)f(m)}{\sqrt{k}}- C T (\log T)^4  \sum_{n\in \mathcal{M}}f(n)^2\right). \] 
\end{lem}

We are now left with the problem of making the first sum on the right-hand side large; the problem of making the right-hand side of \eqref{maximm} large is just the special case when $M=[T]$. In the next session, we will show how to deal with this problem for a wide range of values of $M$.

\section{G\'{a}l-type extremal problems and proof of Theorem~\ref{extreme}}

\subsection{Background on G\'{a}l-type extremal problems}\label{GalBohr}

Before turning to the extremal problems arrived at in the previous section, we would like to place them in context by
describing briefly a line of research that has been instrumental for our approach. This is the study of greatest common divisor (GCD) sums of the form
\begin{equation}\label{gcda}
\sum_{m,n \in \mathcal{M}}\frac{(m,n)^{2\sigma}}{(m
n)^\sigma}
\end{equation}
and the associated (normalized) quadratic forms
\begin{equation}\label{gcdb}
\sum_{m,n\in \mathcal{M}} f(m) f(n) \frac{(m,n)^{2\sigma}}{(m n)^\sigma}\big/ \sum_{n\in \mathcal{M}}f(n)^2,
\end{equation}
where $\mathcal{M}$ is as above and we again assume that $f(n)$ is nonnegative and does not vanish on $\mathcal{M}$. We observe that \eqref{maximm} is smaller than \eqref{gcdb} because the former is obtained from the latter by restricting the sum in the nominator to a subset of $\mathcal{M}\times\mathcal{M}$. In most cases of interest when $1/2\le \sigma <1$, we may obtain a reverse inequality so that the two expressions are of the same order of magnitude. In general, it is clear that if \eqref{gcdb} is large, then also \eqref{maximm} will be large.

The problem is to decide how large either of the two expressions \eqref{gcda} or \eqref{gcdb} can be under the assumption that $|\mathcal{M}|=N$, and more specifically we are interested in the asymptotics when $N\to\infty$ and  $\sigma$ is fixed with $0<\sigma\le 1$.  We refer to \eqref{gcda} and \eqref{gcdb} as G\'{a}l-type sums because the topic begins with a sharp bound of G\'{a}l \cite{G} (of order $CN(\log \log N)^2$) for the growth of \eqref{gcda} when $\sigma=1$.  Dyer and Harman \cite{DH} obtained the first nontrivial estimates for the range $1/2 \le \sigma <1$, and during the past few years, we have reached an essentially complete understanding for the full range $0< \sigma\le 1$, thanks to the papers \cite{ABS, BHS, BS, LR}. The techniques used for different values of $\sigma$ differ considerably, and the problem is particularly delicate for $\sigma=1/2$ at which an interesting ``phase transition'' occurs. We refer to \cite{SS} for an overview of these results and to \cite{ABS, LR} for information about the many different applications of such asymptotic estimates.

It is the insight accumulated in this research that has led to the constructions given below. More specifically, we will follow G\'{a}l \cite{G} when $\sigma$ is close to 1 and \cite{BS2} when $\sigma$ is close to 1/2. The reader will notice that our set $\mathcal{M}$ will contain very smooth numbers when $\sigma$ is close to $1$ in contrast to what happens near $\sigma=1/2$.  Our treatment of the latter case shows that more and more primes are needed when $\sigma$ decreases; the simplest possible choice (made by Aistleitner in \cite{A}) of taking $r$ to be of size $\log N/\log 2$ and the $n_j$ to be the divisors of the square-free number $p_1\cdots p_r$ will be nearly optimal only when $\sigma$ is ``far'' from the endpoints $1$ and $1/2$. Translating this philosophy to Soundararajan's method, we find that the terms picked out in the approximating sum $\sum_{n\le T} n^{-\sigma-it}$ correspond to decreasingly smooth numbers when $\sigma$ goes from $1$ to $1/2$.

\subsection{Levinson's case $\sigma=1$ revisited} It is instructive to consider first the endpoint case $\sigma=1$.
We will now show that
\begin{equation}\label{Gal} \max_{T/2\le t \le T} |\zeta(1+it)|\ge e^{\gamma}\log_2 T + O(\log_3 T). \end{equation}
This estimate is slightly worse than \eqref{Norman} and the best known result of Granville and Soundararajan \cite{GS}, but the benefit is the simplicity of the proof and also that the interval has been shortened. We notice at this point that Hilberdink got the estimate \eqref{Norman} by his version of the resonance method.

We fix a positive number $x$ and an integer $\ell$ (to be determined later) and let $\mathcal{M}$ be the set of divisors of the number
\[ K=K(x,\ell):=\prod_{p\le x} p^{\ell-1}. \]
We require that $K\le [\sqrt{T}]$ and choose $r(n)$ to be the characteristic function of $\mathcal{M}$.
A computation shows that
\[  \sum_{n\in \mathcal{M}, mk=n} \frac{1}{k^\sigma} =\prod_{p\le x} \Big(\ell +\sum_{\nu=1}^{\ell-1} \frac{\ell-\nu}{p^{\nu\sigma}}\Big).\]
Hence
\begin{equation}\label{small}
 \sum_{mk=n}\frac{r(m) r(n)}{k^{\sigma}}\Big/\sum_{n\in \mathcal{M}} r(n)^2
=\prod_{p\le x} \Big(1 +\sum_{\nu=1}^{\ell-1} \Big(1-\frac{\nu}{\ell}\Big)p^{-\nu\sigma}\Big).
 \end{equation}
We now set $\sigma=1$ and find that
\begin{align*} \sum_{mk=n}\frac{r(m) r(n)}{k}\Big/\sum_{n\in \mathcal{M}} r(n)^2
&=\prod_{p\le x} \Big((1-p^{-1})^{-1}+\sum_{\nu=2}^{\ell-1} \frac{\nu}{\ell}p^{-\nu}+O(p^{-\ell }) \Big)\\
&=(1+O(\ell^{-1}))\prod_{p\le x} (1-p^{-1})^{-1} \\
&= \left(1+O(\ell^{-1})+O\Big(\frac{1}{\sqrt{x}\log x}\Big)\right) e^{\gamma}\log x,\end{align*}
where we in the last step used Mertens's third theorem (see \cite{DP} for a precise analysis of the error term).
By the prime number theorem, we may choose $x=(\log T)/(2\log_2 T)$  and $\ell=[\log_2 T]$ if $T$ is large enough. Taking into account \eqref{Sound1} and \eqref{Sound2}, we obtain the desired result \eqref{Gal}.

\subsection{The case $3/4 \le \sigma \le 1-1/\log_2 T$}\label{gal}
We follow the argument of the preceding subsection up to \eqref{small}, from which we deduce that
\begin{align}
 \sum_{mk=n}\frac{r(m) \overline{r(n)}}{k^{\sigma}}\Big/\sum_{n\in \mathcal{M}} |r(n)|^2 \nonumber
& \ge \prod_{p\le x} \Big(1 + \Big(1-\frac{1}{\ell}\Big)p^{-\sigma}\Big) \\
& \ge \prod_{p\le x} \Big(1 + p^{-\sigma}\Big)^{1-\frac{1}{\ell}};  \label{small2} \end{align}
here we used Bernoulli's inequality in the last step. We will use the following lemma to estimate the latter expression.
\begin{lem} \label{psum} There exists an absolute constant $C$ such that
\[ \sum_{p\le x} p^{-\sigma}\ge
                       \sigma \log_2 x +C+ \frac{ x^{1-\sigma}}{(1-\sigma)\log x} \]
whenever $(1-\sigma)\log x \ge 1/2$.
\end{lem}
\begin{proof}
By Abel summation and the inequality $\pi(x)>x/\log x$ which is valid for $x\ge 17$ (see \cite{RSc}), we find that
\begin{equation} \label{primesum} \sum_{p\le x} p^{-\sigma} \ge \frac{x^{1-\sigma}}{\log x}+\sigma \int_2^{x} \frac{dy}{y^{\sigma}\log y}+C', \end{equation}
where $C'$ is an absolute constant. Making the change of variables $u=\log_2 y$, we see that
\begin{align*} \int_2^x \frac{dy}{y^{\sigma}\log y} & =\int_{\log_2 2}^{\log_2 x} e^{(1-\sigma)e^u} du \\
& =\log_2 x -\log_2 2+\sum_{j=1}^{\infty} \frac{(1-\sigma)^j (\log^{j}x-\log^j 2)}{j \cdot j!} \\
&=\log_2 x -\log_2 2+\int_{(1-\sigma)\log 2}^{(1-\sigma)\log x} \frac{e^y-1}{y}dy .\end{align*}
Now using the trivial bound
\[ \int_{0}^a \frac{e^y-1}{y}dy\ge \frac{e^a}{a}-\frac{a+1}{a}   \]
 and returning to \eqref{primesum}, we obtain the desired estimate.
 \end{proof}

We are now prepared to give the first part of the proof of Theorem~\ref{extreme}.

\begin{proof}[Proof of Theorem~\ref{extreme}---part 1] Making the same choices  $x=(\log T)/(2\log_2 T)$  and $\ell=[\log_2 T]$ as in the case $\sigma=1$ and returning to \eqref{small2}, we see that Lemma~\ref{psum} gives that
\begin{align*} \sum_{mk=n}\frac{r(m) \overline{r(n)}}{k^{\sigma}}& \Big/\sum_{n\in \mathcal{M}} |r(n)|^2\\
& \ge \exp\left( \sigma \log_3 T + \frac{2^{\sigma-1}(\log T)^{1-\sigma}}{(1-\sigma)(\log_2 T)^{\sigma}}-E(T,\sigma) \right), \end{align*}
where
\[ E(T,\sigma)  \le C +\frac{(1+\delta)\log_3 T\; (\log T)^{1-\sigma}}{(1-\sigma)(\log_2 T)^{\sigma+1}} \]
for arbitrary $\delta>0$ when $T$ is sufficiently large. Returning to \eqref{plain}, \eqref{Sound1}, and \eqref{Sound2}, we now obtain \eqref{intermediate2} and the desired asymptotic behavior of $\nu(\sigma)$ when $\sigma\nearrow 1$ because
 \[ \frac{\log_3 T}{\log_2 T}\le (1-\sigma)|\log(1-\sigma)| \]
when $\log_2 T\ge e$, by our a priori assumption that $\sigma\le 1-1/\log_2 T$. We also get the uniform lower bound
$\nu(\sigma)\ge 1/(2-2\sigma)$ because $2^{\sigma-1}>1/2$ when $\sigma\ge 1/2$ and $\log_3 T/\log_2 T \to 0$ when $T\to\infty$.
\end{proof}

\subsection{The case $1/2+1/\log_2 T \le \sigma \le 3/4 $}\label{half}
In view of the preceding section, we already know that \eqref{intermediate2} holds for large $T$ when we choose $\nu(\sigma)$ to be an appropriate function bounded below by $1/(2-2\sigma)$. This is just because the interval $[T/2,T]$ is shorter than $[\sqrt{T},T]$ when $T\ge 4$. What remains is therefore to prove that $\nu(\sigma)$  can chosen such that it also has the desired asymptotic  behavior when $\sigma\searrow 1/2$, while \eqref{intermediate2} still holds for large $T$.

We will use a construction from \cite[Section 3]{BS2} which one should understand as a ``reversion'' of an application of the Cauchy--Schwarz inequality in \cite{BS}. This key step in bounding G\'{a}l-type sums from above when $\sigma=1/2$ relies on the existence of so-called divisor closed extremal sets of square-free numbers and a certain completeness property enjoyed by such sets. The interested reader is advised to consult \cite{BS} to see the close connection between our construction and the proof given in that paper.

We recall that, in view of Lemma~\ref{m2lemD}, our goal is to find a multiplicative function $f(n)$ (depending on $\sigma$)  and an associated set of integers $\mathcal{M}$ with $|\mathcal{M}|\le \sqrt{T}$ such that
\begin{equation}\label{toprove}
\sum_{n\in\mathcal{M}, mk=n, k\le M}\frac{f(n)f(m)}{k^\sigma} \ge W(T,\sigma) \times \sum_{n\in \mathcal{M}} f(n)^2 \end{equation}
for suitable values of $M$, depending on $T$, where
\[  W(T,\sigma)=\begin{cases} \exp\left(c\sqrt{\frac{\log T \log_3 T}{\log_2 T}}\right), & \sigma=1/2 \\
\exp\left(\nu(\sigma)\frac{(\log T)^{1-\sigma}}{(\log_2 T)^{\sigma}}\right), & 1/2+1/\log_2 T\le \sigma \le 3/4 \end{cases} \]
and $0<c<1/2$. This was done for $\sigma=1/2$ and $M\ge N^\varepsilon$ for every $\varepsilon>0$ in \cite[Section 3]{BS2}. We will now extend this construction to the range $1/2+1/\log_2 T\le \sigma\le 3/4$, and we will show that we can allow much smaller values of $M$. Since we already obtained the lower bound $\nu(\sigma)\ge 1/(2-2\sigma)$ in the preceding subsection, we will mainly be interested in estimates for $\nu(\sigma)$ when $\sigma$ is sufficiently close to $1/2$.

We begin with the construction of $f(n)$ and $\mathcal{M}$ when $1/2+1/\log_2 T\le \sigma\le 3/4$.  We  follow the scheme in \cite[Section 3]{BS2} word for word, the only essential difference being that we let $P$ be the set of all primes $p$ such that
\[ e\log N\log_2 N< p \le \log N\exp( (2\sigma-1)^{-\alpha})\log_2 N \]
for a suitable $\alpha$, $0<\alpha<1$, and set
\[ f(p):=\frac{(\log N \log_2N)^{1-\sigma}}{\sqrt{|\log(2\sigma-1)|}} \cdot  \frac{1}{p^{1-\sigma}(\log p-\log_2N-\log_3N)},\]
where $N=[\sqrt{T}]$. This defines a multiplicative function $f(n)$, if we require it to be supported on the square-free numbers with prime factors in $P$.
Arguing as in \cite[Section 3]{BS2}, we are now led to consider the quantity
\[
A(N,\sigma):=\frac{1}{\sum_{j\in\N}f(j)^2}\sum_{n\in\N}\frac{f(n)}{n^\sigma}\sum_{d|n}f(d)d^\sigma=\prod_{p\in P}\frac{1+f(p)^2+f(p)p^{-\sigma}}{1+f(p)^2}.
\]
The following estimate, which is the counterpart to \cite[Lemma~1]{BS2} for the case $\sigma=1/2$, is of basic importance.
\begin{lem}\label{lem1}
We have
\begin{equation} \label{annew} A(N,\sigma)\ge \exp\left((\alpha+o(1))\frac{|\log(2\sigma-1)|^{3/2}}{1+|\log(2\sigma-1)|}\frac{(\log N)^{1-\sigma}}{(\log_2 N)^{\sigma}}\right) \end{equation}
uniformly for $1/2+1/\log_2 T \le \sigma \le 3/4$ when $T\to\infty$.
\end{lem}
\begin{proof}
Since $f(p)<1/\sqrt{|\log(2\sigma-1)|}$ for every $p$ in $P$, we find that
\begin{equation}\label{return}
A(N,\sigma)\ge\exp\left((1+o(1))\frac{|\log(2\sigma-1)|}{1+|\log(2\sigma-1)|}\sum_{p\in P}f(p)p^{-\sigma}\right).
\end{equation}
Here and in what follows, the error term goes to $0$
when $T\to\infty$ uniformly for $1/2+1/\log_2 T \le \sigma \le 3/4$.
Now we obtain that
\begin{align*}
\sum_{p\in P}f(p) p^{-\sigma} &=\frac{1}{\sqrt{|\log(2\sigma-1)|}} \cdot (\log N \log_2N)^{1-\sigma} \\& \times \sum_{p\in P}\frac{1}{p(\log p-\log_2N-\log_3N)},
\end{align*}
and
\begin{align*}
\sum_{p\in P}&\ \frac{1}{p(\log p-\log_2N-\log_3N)} \\
& =(1+o(1))\int_{e\log N\log_2 N}^{\log N\exp( (2\sigma-1)^{-\alpha})\log_2 N}\frac{dx}{x\log x(\log x-\log_2N-\log_3N)} \\
& =(1+o(1))\int_{1+\log_2N+\log_3N}^{\log_2N+(2\sigma-1)^{-\alpha}+\log_3N}\frac{dt}{t(t-\log_2N-\log_3N)} \\
& =(\alpha+o(1))\frac{|\log(2\sigma-1)|}{\log_2 N}.
\end{align*}
Returning to \eqref{return}, we obtain the desired estimate \eqref{annew}.
\end{proof}

We proceed next to choose our set $\mathcal{M}$. To this end, we let $P_k$ be the set of all primes $p$ such that $e^k\log N\log_2N<p\le e^{k+1}\log N\log_2N$ for $k=1,\ldots,[(2\sigma-1)^{-\alpha}]$.
Fix $1 < a < 1/\alpha$. Then let $M_k$ be the set of integers that have at least $\frac{a\log N}{k^2|\log(2\sigma-1)|}$ prime divisors in $P_k$,
and let $M'_k$ be the set of integers from $M_k$ that have prime divisors only in $P_k$.
Finally, set
\[ \mathcal{M}:=\supp(f)\setminus\bigcup_{k=1}^{[(2\sigma-1)^{-\alpha}]}M_k.\]

We need to show that we have the bound $|\mathcal{M}|\le N$ if $N$ is large and  $\alpha$ and $a$ have been chosen appropriately. As in \cite[Section 3]{BS2}, we start from the facts that
\begin{equation} \label{bin1}
\binom{m}{n}\le \exp\left(n(\log m-\log n)+n+\log m\right)
\end{equation}
holds when $n\le m$ and $m$ is large enough and that
\begin{equation} \label{bin2}
\frac{\binom{m}{n}}{\binom{m}{n-1}}=\frac{m-n+1}{n} \ge 2
\end{equation}
when $m\ge 3n-1$.
By the prime number theorem, the cardinality of each $P_k$ is at most $e^{k+1}\log N$, and we therefore get, using first
\eqref{bin2} and then \eqref{bin1},
that
\begin{align*}
|\mathcal{M}| & \le\prod_{k=1}^{[(2\sigma-1)^{-\alpha}]}\sum_{j=0}^{\big[\frac{a\log N}{k^2|\log(2\sigma-1)|}\big]} \binom{\big[e^{k+1}\log N\big]}{j}
  \le \prod_{k=1}^{[(2\sigma-1)^{-\alpha}]} 2 \binom{\big[e^{k+1}\log N\big]}{\big[\frac{a\log N}{k^2|\log(2\sigma-1)|}\big]} \\
& \le\exp\Big(\sum_{k=1}^{[(2\sigma-1)^{-\alpha}]}\Big(1+\frac{a \log N \big(k+2+\log|\log(2\sigma-1)|+2\log k\big)}{k^2|\log(2\sigma-1)|}\\ & \qquad \qquad \qquad \qquad \quad +k+1+\log_2N\Big)\Big).
\end{align*}
Hence, choosing $\alpha$ and $a$ suitably, depending on $\sigma$, we have that $|\mathcal{M}|\le N$ for all $N$ large enough. In fact, we notice that the closer $\sigma$ is to $1/2$, the closer to $1$ we can choose $\alpha$ and hence also $a$.

We have now identified the desired set $\mathcal{M}$. The proof of the next lemma shows that we can choose $\alpha$ and $a$ such that $f(n)$ is essentially concentrated on this set.  Here we use the following terminology: A set of positive integers $\mathcal{M}$ is said to be divisor closed if $d$ is in $\mathcal{M}$ whenever $m$ is in $\mathcal{M}$ and $d$ divides $m$. Note that this lemma is the counterpart to \cite[Lemma~2]{BS2} for the case $\sigma=1/2$.
\begin{lem}
\label{lem2} We can choose $\alpha$ depending on $\sigma$, with $\alpha\nearrow 1$ when $\sigma\searrow 1/2$, such that there exists a divisor closed set of integers $\mathcal{M}$ of cardinality at most $N$ and the following estimate holds:
\begin{equation}
\label{i2new}
\frac{1}{\sum_{j\in\N}f(j)^2}\sum_{n\in\N,n\not\in\mathcal{M}}\frac{f(n)}{n^\sigma}\sum_{d|n}f(d)d^\sigma=o(A(N,\sigma)),\quad N\to\infty.
\end{equation}
This estimate is uniform in $\sigma$ for $1/2+1/\log_2 T\le \sigma\le 3/4$.
\end{lem}
\begin{proof}
We use the set $\mathcal{M}$ constructed above. We have already seen that we can in a suitable way let $\alpha\nearrow 1$ when $\sigma\searrow 1/2$. To prove the desired estimate \eqref{i2new}, we begin by noting that
\begin{align} \nonumber
& \frac{1}{A(N,\sigma)\sum_{j\in\N}f(j)^2} \sum_{n\in\N,n\not\in\mathcal{M}}\frac{f(n)}{n^\sigma}\sum_{d|n}f(d)d^{\sigma} \\
&\qquad \qquad \le\frac{1}{A(N,\sigma)\sum_{j\in\N}f(j)^2}\sum_{k=1}^{[(2\sigma-1)^{\alpha}]}\sum_{n\in M_k}\frac{f(n)}{n^{\sigma}}\sum_{d|n}f(d)d^{\sigma}. \label{i2.5}
\end{align}
Now for each $k=1,\ldots,[(2\sigma-1)^{-\alpha})]$, we have that
\begin{align} \nonumber
&\frac{1}{A(N,\sigma)\sum_{j\in\N}f(j)^2}\sum_{n\in M_k}\frac{f(n)}{n^{\sigma}}\sum_{d|n}f(d)d^{\sigma} \\ & \qquad \quad=\frac{1}{\prod_{p\in P_k}(1+f(p)^2+f(p)p^{-\sigma})}\sum_{n\in M'_k}\frac{f(n)}{n^{\sigma}}\sum_{d|n}f(d)d^{\sigma} \nonumber \\
\label{i3}
&\qquad  \quad \le\frac{1}{\prod_{p\in P_k}(1+f(p)^2)}\sum_{n\in M'_k}f(n)^2\prod_{p\in P_k}\left(1+\frac{1}{f(p)p^{\sigma}}\right).
\end{align}
To deal with the product to the right in \eqref{i3}, we make the following computation:
\begin{align}
\nonumber
\prod_{p\in P_k}& \left(1+\frac{1}{f(p)p^{\sigma}}\right)\\& =\prod_{p\in P_k}\left(1+\frac{\sqrt{|\log(2\sigma-1)|}}{(\log N\log_2 N)^{1-\sigma}}p^{1-2\sigma}(\log p-\log_2N-\log_3N)\right) \nonumber
\\
\nonumber
& \le\left(1+(k+1)e^{k(1-2\sigma)}\sqrt{|\log(2\sigma-1)|}(\log N\log_2 N)^{-\sigma}\right)^{e^{k+1}\log N}\\
&\le \exp\left((k+1)e^{k(2-2\sigma)+1}\sqrt{|\log(2\sigma-1)|}(\log N)^{1-\sigma}(\log_2 N)^{-\sigma} \right)\nonumber\\
&=\exp\left(o\left(\frac{\log N}{|\log(2\sigma-1)|}\right)\frac{1}{k^2}\right), \label{i4}
\end{align}
where the latter relation holds simply because $k\le (2\sigma-1)^{-\alpha}$. Since every number in $M_k'$ has at least
$\frac{a\log N}{k^2 |\log(2\sigma-1)|}$ prime divisors and $f(n)$ is a multiplicative function, it  follows that
\[
\sum_{n\in M'_k}f(n)^2\le   b^{-a\frac{\log N}{k^2|\log(2\sigma-1)|}}\prod_{p\in P_k}(1+bf(p)^2)
\]
whenever $b> 1$ and hence
\begin{equation}
\label{i5}
\frac{\sum_{n\in M'_k}f(n)^2}{\prod_{p\in P_k}(1+f(p)^2)}\le b^{-a\frac{\log N}{k^2|\log(2\sigma-1)|}}\exp\left(\sum_{p\in P_k}(b-1) f(p)^2\right).
\end{equation}
Finally,
\begin{align*}
\sum_{p\in P_k} f(p)^2& =\frac{(\log N \log_2N)^{(2-2\sigma)}}{|\log(2\sigma-1)|}\sum_{p\in P_k}\frac{1}{p^{2-2\sigma}(\log p-\log_2N-\log_3N)^2}
\\
& \le (1+o(1))\frac{(\log N \log_2N)^{(2-2\sigma)}}{|\log(2\sigma-1)|}\int_{e^k\log N\log_2 N}^{e^{k+1}\log N\log_2 N}\frac{dx}{k^2x^{2-2\sigma}\log x} \\
& \le (1+o(1))\frac{\log N}{k^2|\log(2\sigma-1)|}\frac{e^{2\sigma-1}-1}{2\sigma-1}e^{k(2\sigma-1)}.
\end{align*}
Combining the last inequality with~\eqref{i5} and~\eqref{i4}, we get that~\eqref{i3} is at most
\[
\exp\left(\Big(\frac{e^{2\sigma-1}-1}{2\sigma-1}e^{k(2\sigma-1)}(b-1)-a \log b +o(1)\Big)\frac{\log N}{k^2|\log(2\sigma-1)|}\right).
\]
We see that when $\sigma$ is close to $1/2$, the factor in front of $b-1$ is close to $1$. In this case, we may therefore choose both $\alpha$ and $a$ close to $1$ and then $b$ close to $1$, to arrange it so that
\[ (b-1)\frac{e^{2\sigma-1}-1}{2\sigma-1}e^{k(2\sigma-1)}-a \log b < 0; \]
the latter inequality can of course be obtained trivially for all values of $\sigma$ if we choose $\alpha$, $a$, $b$ appropriately. Returning to~\eqref{i2.5}, we therefore see that the desired relation~\eqref{i2new} has been established, as well as the asymptotic relation between $\sigma$ and $\alpha$.
\end{proof}

It remains to make the additional restriction $k\le M$ in \eqref{toprove}. The next lemma addresses this point and proves a result which is much stronger than what we need to finish the proof of Theorem~\ref{extreme}. In fact, this lemma is what we would need to prove the counterpart to Theorem~\ref{part} for $1/2+1/\log_2 T<\sigma \le 3/4.$
\begin{lem}
\label{lem8}
Let $\mathcal{M}$ be as defined above. Then
\begin{equation}
\label{eq8}
\frac{1}{\sum_{j\in\N}f(j)^2}\sum_{n\in\mathcal{M}}\frac{f(n)}{{n^\sigma}}\sum_{d|n,\,d\le n/M}f(d){d^\sigma}=o(A(N,\sigma)),\quad N\to\infty
\end{equation}
uniformly for $1/2+1/\log_2 T\le \sigma \le 3/4$, where
\[ M:=\exp\big(e (\sqrt{|\log(2\sigma-1)|}+3)(\log N \log_2 N)^{1-\sigma}\big). \]
\end{lem}
We notice that here we only need that $1<a<1/\alpha$. It may also be observed that with some extra effort one may replace $e$ by a somewhat smaller constant  in the definition of $M$.
\begin{proof}[Proof of Lemma~\ref{lem8}]
To begin with, we observe that
\[
\sum_{n\in\mathcal{M}}\frac{f(n)}{n^\sigma}\sum_{d|n,\,d\le n/M}f(d)d^\sigma=
\sum_{n\in\mathcal{M}}f(n)^2\sum_{k|n,\,k\ge M}\frac{1}{f(k)k^\sigma}.
\]
It is therefore enough to show that for each $n$ in $\mathcal{M}$ we have
\[
\sum_{k|n,\,k\ge M}\frac{1}{f(k)k^\sigma}=o(A(N,\sigma)),\quad N\to\infty.
\]
Multiplying and dividing the $k$th term by $k^{-\delta}$ and using that $f(k)$ is a multi- plicative function\footnote{This estimation technique is known as Rankin's trick.},  we deduce that
\begin{align*}
\sum_{k|n,\,k\ge M}\frac{1}{f(k)k^\sigma} & \le M^{-\delta}\prod_{ p|n}\left(1+\frac{1}{p^{\sigma-\delta} f(p)}\right) \\
& \le M^{-\delta}
\exp\Big((1+o(1))\max_{p|n} p^{\delta}\sum_{p|n} \frac{1}{p^\sigma f(p)}\Big)
\end{align*}
whenever $\delta>0$.
We find that by the definition of $\mathcal{M}$,
\begin{align*} \sum_{p|n} \frac{1}{p^\sigma f(p)} &  =
\sum_{p|n} \frac{\sqrt{|\log(2\sigma-1)|}}{(\log N\log_2 N)^{1-\sigma}}p^{1-2\sigma}(\log p-\log_2N-\log_3N)\\
&\le \sum_{k=1}^{[(2\sigma-1)^{-\alpha}]}
\frac{a\log N}{k^2|\log(2\sigma-1)|}\frac{\sqrt{|\log(2\sigma-1)|}}{(\log N\log_2 N)^{1-\sigma}} \\
& \qquad \qquad \qquad \times (e^{k+1}\log N\log_2 N)^{1-2\sigma}(k+1) \\
& \le a\alpha(\sqrt{|\log(2\sigma-1)|}+3)\frac{(\log N)^{1-\sigma}}{(\log_2 N)^{\sigma}}. \end{align*}
We now set $\delta=1/\log_2 N$ and 
obtain
\[ \sum_{k|n,\,k\ge M}\frac{1}{f(k)\sqrt{k}} \le
\exp\left(\big(ae \alpha-e\big)(\sqrt{|\log(2\sigma-1)|}+3)\frac{(\log N)^{1-\sigma}}{(\log_2 N)^{\sigma}}\right)<1, \]
provided that $a\alpha<1$.
\end{proof}
We are now finally prepared to finish the proof of Theorem~\ref{part}.
\begin{proof}[Proof of Theorem~\ref{extreme}---part 2]
We recall that $N=[\sqrt{T}]$. This means that we need to prove that \eqref{toprove} holds for $M=T\ge N^2$ and suitable choices of the parameters $\alpha$ and $a$, ensuring the desired asymptotic behavior
\[ \nu(\sigma)=(1/\sqrt{2}+o(1))\sqrt{|\log(2\sigma-1)|} \]
when $\sigma\searrow 1/2$. We conclude by observing that this follows from the three preceding lemmas. \end{proof}

\section{Proof of Theorem~\ref{part}}

All the work needed for the proof of Theorem~\ref{part} has now been made. Indeed, we may use the construction in \cite[Section 3]{BS2} and  the estimates established there, corresponding to Lemma~\ref{lem1} and Lemma~\ref{lem2}. We only need a minor modification of Lemma~\ref{lem8}, which is as follows.

\begin{lem}
\label{lem9}
Let $f(n)$ and $\mathcal{M}$ be as defined above in the case $\sigma=1/2+1/\log_2 T$.  Then
\[
\frac{1}{\sum_{j\in\N}f(j)^2}\sum_{n\in\mathcal{M}}\frac{f(n)}{{\sqrt{n}}}\sum_{d|n,\,d\le n/M}f(d){\sqrt{d}}=o(A(N,1/2)),\quad N\to\infty,
\]
where
\[ M:=\exp\big(e (\sqrt{\log N \log_2 N\log_3 N}\big). \]
\end{lem}

The proof of Lemma~\ref{lem9} is word for word the same as that of Lemma~\ref{lem8} and is therefore omitted. The desired estimate for $\widetilde{\widetilde{M}}_2(R,T)$ (see Lemma~\ref{m2lemD}) is now obtained in exactly the same way as in the preceding case when $\sigma=1/2+1/\log_2 T$.

\section{Concluding remarks}

To obtain a more precise estimate in the range $\sigma_0\le \sigma \le 1-1/\log_2 T$ for a suitable $\sigma_0$, $1/2<\sigma_0<1$, we could combine the two constructions in the range in which the powers of $\log T$ in Lemma~\ref{m1lemma} and Lemma~\ref{m2lem} are insignificant, say when $\sigma_0\le \sigma \le 1-1/\sqrt{\log_2 T}$. This can be done as follows. Let $\ell$ and $x$ be two positive integers such that $N:=\ell^{\pi(x)}$ satisfies the inequality $N\le \sqrt{T}$; here $\pi(x)$ is as usual the number of primes $\le x$. We let again $\mathcal{M}$ be the set of divisors of the number $\prod_{p\le x} p^{\ell-1}$  and choose $f(n)$ to be the characteristic function of $\mathcal{M}$.  We observe that the only difference from Subsection~\ref{gal} is that we have replaced the condition $\max \mathcal{M}\le \sqrt{T}$ by the less severe requirement that $|\mathcal{M}|\le \sqrt{T}$. The computation is precisely as in Subsection~\ref{gal}, but we are now free to choose a larger $x$ and consequently  a smaller $\ell$. From Lemma~\ref{psum}, we see that this should be done so that we make
\[ \frac{x^{1-\sigma}}{(1-\sigma)\log x}  \]
as large as possible. A calculus argument shows that $\ell$ should be of order $1/(1-\sigma)$ and  consequently
$x$ of order $(1-\sigma) \log_3 T$. We see again the phenomenon that more and more primes are used when $\sigma$ decreases.

Our final remark is about what we might expect the true growth of $|\zeta(\sigma+it)|$ to be. Farmer, Gonek, and Hughes \cite{FGH} conjectured, appealing to random matrix theory, that
\[ \max_{1\le t \le T} |\zeta(1/2+it)|= \exp\left(\Big(\frac{1}{\sqrt{2}}+o(1)\Big)\sqrt{\log T \log_2 T}\right) \]
and in \cite[Remark 2]{La}, it is suggested that for example
\[  \max_{T/2\le t \le T}  |\zeta(1/2+1/\log_2 T+it)|
=\exp\left(\Big(c+o(1)\Big)\sqrt{\log T \log_2 T}\right) \]
for some $c<1/\sqrt{2}$. Hence the asymptotics of our estimates when $\sigma\searrow 1/2$ is an order of magnitude smaller than this prediction. On the other hand, it is expected that the true growth
of   $\max_{1\le t \le T} |\zeta(\sigma+it)|$ is of order $\exp((c+o(1))(\log T)^{1-\sigma}(\log_2 T)^{-\sigma})$ for some constant $c$ when $1/2<\sigma< 1$, and the asymptotics of Theorem~\ref{extreme} when $\sigma\nearrow 1$ is indeed consistent with what is predicted in \cite[Remark 2]{La}.

The predictions of \cite{FGH} and \cite{La} are however very far from the known upper bounds for the growth of $\zeta(1/2+i t)$. We refer to Bourgain's recent paper \cite{B} for the best result when $\sigma=1/2$:
\[ |\zeta(1/2+it)| \le C_{\varepsilon} |t|^{13/84+\varepsilon} \]
for every $\varepsilon>0$.

\subsubsection*{Acknowledgment} We are grateful to Maksym Radziwi{\l\l} for helpful remarks regarding  Lamzouri's paper \cite{La}.


\begin{thebibliography}{1}
\bibitem{A} C. Aistleitner, \emph{Lower bounds for the maximum of the Riemann zeta function along vertical lines}, Math. Ann. \textbf{365} (2016), 473--496.

\bibitem{ABS} C. Aistleitner, I. Berkes, and K. Seip, \emph{GCD sums from Poisson integrals and systems of dilated functions}, J. Eur. Math. Soc. \textbf{17} (2015), 1517--1546.

\bibitem{BR} R. Balasubramanian and K. Ramachandra,
\emph{On the frequency of Titchmarsh's phenomenon for $\zeta(s)$. III},
Proc. Indian Acad. Sci. Sect. A \textbf{86} (1977), 341--351.

\bibitem{BHS} A.~Bondarenko, T.~Hilberdink, and K.~Seip,
\emph{G\'{a}l-type GCD sums beyond the critical line}, J. Number Theory \textbf{166} (2016), 93--104.

\bibitem{BS}
A. Bondarenko and K. Seip, \emph{GCD sums and complete sets of square-free numbers}, Bull.
London Math. Soc. \textbf{47} (2015), 29--41.

\bibitem{BS2}
A. Bondarenko and K. Seip, \emph{Large GCD sums and extreme values of the Riemann zeta function}, Duke Math. J., to appear; arXiv:1507.05840.

\bibitem{B} J. Bourgain, \emph{Decoupling, exponential sums and the Riemann zeta function}, J. Amer. Math. Soc.  \textbf{30} (2017), 205--224.

\bibitem{DP} H.~G.~Diamond and J. Pintz, \emph{Oscillation of Mertens' product formula} (English,
J. Th\'{e}or. Nombres Bordeaux \textbf{21} (2009), 523--533.

\bibitem{DH} T. Dyer and G. Harman, \emph{Sums involving common divisors}, J. London Math. Soc. \textbf{34} (1986), 1--11.

\bibitem{FGH} D. W. Farmer, S. M. Gonek, and C. P.  Hughes,
\emph{The maximum size of $L$-functions},
J. Reine Angew. Math. \textbf{609} (2007), 215--236.

\bibitem{G}
I. S. G\'{a}l, \emph{A theorem concerning Diophantine approximations},
Nieuw Arch. Wiskunde \textbf{23} (1949), 13--38.

\bibitem{GS} A. Granville and K. Soundararajan,
\emph{Extreme values of $|\zeta(1+it)|$},  ``The Riemann Zeta Function and Related Themes: Papers in Honour of Professor K. Ramachandra'',
pp. 65--80, Ramanujan Math. Soc. Lect. Notes Ser., 2, Ramanujan Math. Soc., Mysore, 2006.





\bibitem{Hi} T. Hilberdink, \emph{An arithmetical mapping and applications to $\Omega$-results for the Riemann zeta function}, Acta Arith. \textbf{139} (2009), 341--367.


\bibitem{La}  Y. Lamzouri, \emph{On the distribution of extreme values of zeta and L-functions in the strip $1/2<\sigma<1$}, Int. Math. Res. Not. IMRN \textbf{2011}, 5449--5503.

\bibitem{L} N. Levinson, \emph{$\Omega$-theorems for the Riemann zeta-function}, Acta Arith. \textbf{20} (1972), 317--330.

\bibitem{LR} M. Lewko and M. Radziwi{\l}{\l}, \emph{Refinements of G\'{a}l’s theorem and applications}, Adv. Math. \textbf{305} (2017), 280--297.






\bibitem{M} H. L. Montgomery, \emph{Extreme values of the Riemann zeta function},  Comment. Math. Helv. \textbf{52} (1977), 511--518.

\bibitem{QQ} H.~Queff\'elec and M.~Queff\'elec,
\newblock  \emph{Diophantine Approximation and Dirichlet series}.
 \newblock  HRI Lecture Notes Series - 2, Hindustan Book Agency, New Delhi, 2013.

\bibitem{RS} K. Ramachandra and A. Sankaranarayanan, \emph{Note on a paper by H. L. Montgomery (Omega theorems for the Riemann zeta-function)}, Publ. Inst. Math. (Beograd) (N.S.) \textbf{50} (\textbf{64}) (1991), 51--59.

\bibitem{RSc} J.~B.~Rosser and L. Schoenfeld, \emph{Approximate formulas for some functions of prime numbers}, Illinois J. Math. \textbf{6} (1962), 64--97.

\bibitem{SS} E. Saksman and K. Seip, \emph{Some open questions in analysis for Dirichlet series},
 in ``Recent Progress on Operator Theory and Approximation in Spaces of Analytic Functions'', pp. 179--193,  Contemp. Math. \textbf{679}, Amer. Math. Soc., Providence RI, 2016.

\bibitem{S} K. Soundararajan, \emph{Extreme values of zeta and $L$-functions},
Math. Ann. \textbf{342} (2008), 467--486.


\bibitem{T} E. C. Titchmarsh, \emph{The Theory of the Riemann Zeta-Function}, 2nd Edition, Oxford University Press, New York, 1986.

\end{thebibliography}
\end{document}